\pdfoutput=1
\documentclass{amsart}
 \usepackage{booktabs}
\usepackage{amssymb}
\usepackage[utf8]{inputenc}
\usepackage{longtable, geometry}
\usepackage[english]{babel}
\makeindex

\usepackage{hyperref}
\hypersetup{
colorlinks=true,       
    linkcolor=blue,          
    citecolor=blue,        
    filecolor=blue,      
    urlcolor=blue           
}

\usepackage{rotating,tabularx}

\usepackage[all]{xy}
\usepackage{tikz}
\usepackage{todonotes}
\usepackage[alphabetic,backrefs,lite]{amsrefs}
\usepackage{enumerate}
\usepackage{verbatim}
\usepackage{multirow}
\usepackage{pgf,tikz}
\usepackage{float}
\usepackage{rotating}
\usetikzlibrary{arrows}
\usepackage{listings}

\CompileMatrices

\newtheorem{theorem}{Theorem}[section]
\newtheorem{lemma}[theorem]{Lemma}

\newtheorem{proposition}[theorem]{Proposition}
\newtheorem{corollary}[theorem]{Corollary}
\theoremstyle{definition}
\newtheorem{definition}[theorem]{Definition}
\newtheorem{example}[theorem]{Example}

\newtheorem{remark}[theorem]{Remark}

\def\cc{{\mathbb C}}

\def\zz{{\mathbb Z}}

\def\nn{{\mathbb N}}
\def\qq{{\mathbb Q}}
\def\pp{{\mathbb P}}
\def\ff{{\mathbb F}}
\def\Osh{{\mathcal O}}

\def\Eff{\operatorname{Eff}}

\def\Cl{\operatorname{Cl}}

\def\Nef{\operatorname{Nef}}

\def\BNef{\operatorname{BNef}}
\def\BEff{\operatorname{BEff}}

\newcommand{\oo}{\mathcal{O}}

\begin{document}

\title{Cox rings of  K3 surfaces of Picard number three}

\author{Michela Artebani}
\address{
Departamento de Matem\'atica, \newline
Universidad de Concepci\'on, \newline
Casilla 160-C,
Concepci\'on, Chile}
\email{martebani@udec.cl}

\author{Claudia Correa}
\address{
Departamento de Matem\'atica, \newline
Universidad de Concepci\'on, \newline
Casilla 160-C,
Concepci\'on, Chile}
\email{claucorrea@udec.cl}

\author{Antonio Laface}
\address{
Departamento de Matem\'atica, \newline
Universidad de Concepci\'on, \newline
Casilla 160-C,
Concepci\'on, Chile}
\email{alaface@udec.cl}

\thanks{The authors have been partially 
supported by Proyecto FONDECYT Regular 
N. 1160897  (first and second author),
Proyecto FONDECYT Regular N. 1190777 (third author)
and  Proyecto Anillo ACT 1415 PIA Conicyt. 
The second author has been supported by CONICYT PCHA/DoctoradoNacional/2012/21120687.}

\begin{abstract}
Let $X$ be a projective K3 surface over $\cc$. 
We prove that its Cox ring $R(X)$ 
has a generating set whose degrees  
are either classes of smooth rational curves, 
sums of at most three elements of the Hilbert basis 
of the nef cone, or of the form $2(f+f')$, 
where $f,f'$ are classes of elliptic fibrations with $f\cdot f'=2$.
This result and  techniques using Koszul's type 
exact sequences allow to determine a 
generating set for the Cox ring of 
all Mori dream K3 surfaces of Picard number three
which is minimal in most cases.
A presentation for the Cox ring is given 
in some special cases with few generators.
\end{abstract}
\maketitle

\tableofcontents

\section*{Introduction}

The Cox ring of a normal projective variety $X$ 
defined over the complex numbers with finitely generated and free 
divisor class group $\Cl(X)$ is the graded algebra \cite{A.D.H.L}
\[
R(X):=\bigoplus_{[D]\in \Cl(X)}H^0(X,\oo_X(D)).
\]
The variety $X$ is called Mori dream space when the Cox ring is finitely generated.
Important examples of Mori dream spaces are toric varieties, 
whose Cox ring is a polynomial ring.
In this context there are two main problems: 
determine conditions on $X$ such that $R(X)$ is finitely generated,
and find an explicit presentation for $R(X)$. 
An important property of Mori dream spaces is  
that any such variety is a GIT quotient of an open Zariski subset 
of an affine space, the spectrum of $R(X)$, by the action of a quasitorus.
This allows to define homogeneous coordinates on $X$,
as in the case of the projective space, and allows a combinatorical 
approach to certain geometric and arithmetic properties of $X$, 
as in the case of 
toric varieties \cite{A.D.H.L}.

This paper deals with Cox rings of K3 surfaces.
In \cite{A.H.L} and \cite{JM} the authors proved independently 
that the Cox ring of a K3 surface is finitely generated 
if and only if its effective cone is polyhedral, 
or equivalently if its automorphism group is finite. 
K3 surfaces with this property have been classified 
in \cite{VN1,VN,VN2,P.S, EV1} 
(see also \cite[\S 5.1.5]{A.D.H.L}).
The main purpose of the paper is to develop computational tools to compute 
Cox rings of K3 surfaces.  To achieve this goal, we first extend the known 
techniques for finding generators of the Cox ring.
This allows to prove a general theorem on Cox rings of K3 surfaces (Theorem \ref{gen}): 
the degrees of the generators are either classes of $(-2)$-curves, nef classes which are sums of at most three 
elements of the Hilbert basis of the nef cone (allowing repetitions),
or classes of the form $2(f+f')$, where $f,f'$ are classes of fibers of 
elliptic fibrations with $f\cdot f'=2$.
Afterwards, we apply this result and other techniques 
based on Koszul exact sequences
to determine the degrees a generating set for 
the Cox rings of general elements
of the $26$ families of Mori dream K3 surfaces 
of Picard number three (Theorem \ref{main}). 
The proof is obtained by means of several
computational programs implemented in MAGMA \cite{B.C.P},
which allow first to compute the effective and nef cone 
of the K3 surface, then to find the degrees of a set of generators 
of $R(X)$, and finally to check the minimality of such set (see Section \ref{magma}).
In some special cases, when the Cox ring has few generators,
we are able to provide a presentation for the Cox ring (see Section \ref{special}).

\section{Preliminaries}
We will work over the field $\cc$ of complex numbers.

\subsection{Cox rings}

In this section we will give the necessary preliminaries on Cox rings (see \cite{A.D.H.L}). 

\begin{definition}\label{defcox}
Let $X$ be a normal projective variety defined over $\cc$ 
and assume that the divisor class group $\Cl(X)$ is finitely generated and free. 
The \emph{Cox ring} of $X$ is defined as
\[
R(X):=\bigoplus_{D\in K}H^0(X,\oo_X(D)),
\]
where $K\subseteq \rm{WDiv}(X)$ is a subgroup of the group 
of Weil divisors such that the canonical morphism $K\to \Cl(X)$, 
which associates to $D$ its class $[D]$, is an isomorphism.
\end{definition}

We observe that $R(X)$ is a $K$-graded algebra over $\cc$, 
i.e. it has a direct sum decomposition into complex vector spaces 
\[
R(X)_D:=H^0(X,\oo_X(D)),
\] 
where $D\in K$, such that
$ R(X)_{D_1}\cdot R(X)_{D_2}\subseteq R(X)_{D_1+D_2}.$
We will say that $f\in R(X)$ is {\em homogeneous} 
if it belongs to $R(X)_D$ for some $D\in K$ and in this case 
we define its {\em degree} to be ${\rm deg}(f) = [D]$.

A set of homogeneous generators $\{f_i: i \in I\}$ of $R (X)$ is a {\em minimal generating set} 
when each $f_i$ can't be expressed as a polynomial in the remaining elements $f_j$. 
Moreover, we say that $R (X)$ {\em has a generator in degree} $w\in \Cl (X)$ 
if each minimal generating set of $R (X)$ contains a nontrivial element of $R (X)_D$, where $[D]=w$.\\
In the following, given a divisor $A$, we denote by $R (X)_A$ the vector space $R (X)_D$ 
where $D\in K$ and $[D] = [A]$.

\begin{definition}
A variety $X$ as in Definition \ref{defcox} is called
 \emph{Mori dream space} if it has finitely generated Cox ring $R(X)$. 
\end{definition}

We denote by ${\rm Eff}(X)\subseteq \Cl(X)_{\qq}$ 
the {\em effective cone} of $X$, that is the cone generated by 
the classes of effective divisors. 
For the following result see \cite[Proposition 2.1]{A.H.L}.

\begin{proposition}\label{gen}
Let $f_i,\ i\in I$ be a set of homogeneous generators of $R(X)$. Then 
 \[
 {\rm Eff}(X)={\rm cone}(\deg(f_i): i\in I).
 \]
In particular, the effective cone of a Mori dream space is polyhedral. 
\end{proposition}

\begin{corollary}\label{gendegree}
If $D$ is an effective divisor such that $[D]$ is an element of the Hilbert basis of ${\rm Eff}(X)\subseteq \Cl(X)_{\qq}$, 
then the Cox ring $R(X)$ has a generator in degree $[D]$.
\end{corollary}

\begin{proof}
Let $\{f_i: i\in I\}$ be a minimal generating set of $R(X)$.
Since $[D]$ belongs to the Hilbert basis of the effective cone, then $[D]=\deg(f_i)$ for some $i\in I$
by Proposition \ref{gen}.
\end{proof}

For example, if $X$ is a smooth projective surface and $D$ is an integral effective divisor on $X$ with $D^2 < 0$, 
then $R(X)$ has a generator in degree $[D]$.

\subsection{Linear systems on K3 surfaces}
We now recall some classical results on linear systems on K3 surfaces.
We start with a result about the base locus.

\begin{proposition}\label{lsk3}
Let $X$ be a smooth projective K3 surface and $D$ be a non-zero 
effective divisor on $X$. Then
\begin{enumerate}
\item $|D|$ has no base points outside its fixed components;
\item if $D$ is nef, then $|D|$ is base point free unless there exist 
two curves $E,F$ and an integer $k\geq 2$ with 
\[
D\sim kF+E,\ F^2=0,\ E^2=-2,\ F\cdot E=1;
\]
\item if $D$ is nef, then $h^1(X,\Osh_X(D))=0$ unless $D\sim kF$
where $F$ is a primitive divisor with $F^2=0$ and $k\geq 2$.
\end{enumerate}
\end{proposition}
\begin{proof}
(i) see \cite[Corollary 3.2]{SD}, (ii) follows from \cite[Section 2.7]{SD} and (iii) see \cite{K.L}.
\end{proof}

\begin{corollary}\label{section}
Let $X$ be a K3 surface such that none of its elliptic fibrations has a section.
Then any nef divisor on $X$ is base point free.
\end{corollary}

We recall that a non-empty linear system $|D|$ is called {\em hyperelliptic} if 
its generic member is a hyperelliptic curve. 
The following result follows from \cite[Proposition 2.6]{SD}.

\begin{proposition}\label{hyp}
 A nef divisor $D$ on a K3 surface with $D^2>0$ is hyperelliptic  if and only either $D^2=2$,
 or there is a smooth elliptic curve $E$ such that  $D\cdot E =2$, or $D\sim 2B$ 
 for a smooth curve $B$ with $B^2=2$.
\end{proposition}

 \begin{proposition}\label{ray} 
$($\cite[Proposition 3.4]{A.H.L}$)$
Let $X$ be a projective K3 surface and $D$ be an effective divisor on $X$.
\begin{itemize}
\item[i)] If $D$ is not hyperelliptic, then the ray $R_{[D]}:=\bigoplus_{n\in \nn} H^0(X,\Osh_X(nD))$ is generated in degree one.
\item [ii)]If $D^2=2$, then $R_{[D]}$ is generated in degrees one and three.
\item [iii)]If $D$ is hyperelliptic and $D^2>2$, then $R_{[D]}$ is generated in degrees one and two.
\end{itemize}
\end{proposition}

\subsection{Cox rings of K3 surfaces}

The following theorem \cite[Theorem 5.1.5.1]{A.D.H.L} characterizes Mori dream K3 surfaces.

\begin{theorem}\label{Mori dream}
Let $X$ be an algebraic K3 surface. Then the following statements are equivalent.
\begin{enumerate}
\item $X$ is a Mori dream surface.
\item The effective cone $\Eff(X)\subseteq \Cl_{\qq}(X)$ is polyhedral.
\item The automorphism group of $X$ is finite.
\end{enumerate}
\end{theorem}
Moreover, if the Picard number is at least three, 
then $(i)$ is equivalent to the property that $X$ contains only finitely many smooth rational curves. 
In this case, these curves are $(-2)$-curves 
and their classes generate the effective cone (see \cite[Remark 7.2]{Kovacs}).
This result  allows \cite[Theorem 5.1.5.3]{A.D.H.L} 
to classify Mori dream K3 surfaces using the classification 
of K3 surfaces with finite automorphism group (see, \cite{VN1,VN,VN2,P.S,EV1}).

\begin{theorem}
\label{classification}
Let $X$ be an algebraic K3 surface with 
Picard number $\rho(X)$. Then $X$ is 
Mori dream if and only if one of the 
following occurs:
\begin{enumerate}
\leftskip -5mm
\item 
$\rho(X)=2$ and $\Cl(X)$ contains a class
$w$ with $w^2\in\{-2,0\}$.
\item 
$\rho(X)=3$ and $\Cl(X)$
is isometric to one of the following $26$
lattices:
\begin{align*}
S'_{4,1,2}&=\langle 2e_1+e_3,e_2,2e_3\rangle,\\
S_{k,1,1}&=\langle ke_1,e_2,e_3\rangle,\ k\in\{4,5,6,7,8,10,12\},\\ 
S_{1,k,1}&=\langle e_1,ke_2,e_3\rangle,\ k\in\{2,3,4,5,6,9\},\\
S_{1,1,k}&=\langle e_1,e_2,ke_3\rangle,\ k\in\{1,2,3,4,6,8\},
\end{align*}
where the intersection matrix of $e_1,e_2,e_3$ is
$
 \left[
 \begin{array}{rrr}
  -2 & 0 & 1\\
   0 & -2 & 2\\
  1 & 2 & -2
 \end{array}
 \right]
$,
and
$$
S_1=(6)\oplus 2A_1,
\qquad
S_2=(36)\oplus A_2,\qquad S_{3}=(12)\oplus A_2, 
$$
$$
S_4=\left[
\begin{array}{rrr}
6&0&-1\\
0&-2&1\\
-1&1&-2
\end{array}
\right],  \qquad S_{5}=(4)\oplus A_2,
\qquad
S_6=\left[
\begin{array}{rrr}
14&0&-1
\\
0&-2&1
\\
-1&1&-2
\end{array}
\right].
$$

\item 
$\rho(X)=4$ and $\Cl(X)$ is 
isometric to one of the $14$
lattices in \cite{EV1}.

\item 
$5\leq \rho(X) \leq 19$ and
$\Cl(X)$ is isometric to a list of 78
lattices $($see the list in \cite[Theorem 5.1.5.3 (iv)]{A.D.H.L}$)$.
\end{enumerate}
\end{theorem}

In \cite{A.H.L} the authors  determined the Cox ring of all K3 surfaces 
with a non-symplectic involution when $2\leq \rho(X) \leq 5$. 
Moreover, they compute the Cox ring of K3 surfaces that are general double covers of  del Pezzo surfaces.
 In \cite{JCO} the author gives a new proof of the finite generation 
 of the Cox ring when the effective cone is rational polyhedral 
 and develops a technique that allows to calculate the degrees 
 of generators and  relations for the Cox rings of several examples 
 for $\rho(X)=2$, among which:  quartic surfaces that contain a line,  
 quartic surfaces that contain two plane conics and double covers of 
 the Hirzebruch surface $\ff_4$.

\section{Computing a generating set of the Cox ring}
\label{techniques}

\subsection{Koszul type sequences}\label{techniques koszul}
In this section we will present some techniques which allow to show 
that the Cox ring of a projective variety has no generators in a certain degree.	
We will use the standard   notations 
$H^i(X, D) = H^i (X, \oo_X(D))$ and $h^i(X, D) = h^i(X, \oo_X (D))$. 

 \begin{theorem}\label{koszul3}
Let $X$ be a smooth projective variety over $\cc$, 
$E_1,E_2,\dots, E_n$ be effective divisors of $X$
such that $\bigcap_{i=1}^nE_i=\emptyset$ 
and $s_{i}\in H^0(X, E_i)$ be a defining section for $E_i$, $i=1,\dots,n$. 
Let
\[
\mathbb K_0(D):=\Osh_X,\ \mathbb 
K_i(D):=\bigoplus_{1\leq j_1<\dots<j_i\leq n}\Osh_X(-E_{j_1}-\cdots-E_{j_i}),\ i=1,\dots,n.
\]
Then there is an exact sequence of sheaves:
\begin{equation}\label{koszul}
\xymatrix{
 0  \ar[r] &  \mathbb K_n(D)\ar[r]^{d_n} & \mathbb K_{n-1}(D) 
 \ar[r]^{d_{n-1}} & \cdots \ar[r]& \mathbb K_1(D)\ar[r]^{d_1} & \mathbb K_0(D)\ar[r] & 0,
 }
\end{equation}
where $d_1(u_j)=s_ju_0$ para $j=1,\dots,n$ and 
\[
d_i(u_{j_1\cdots j_i})=\sum_{r=1}^i (-1)^{r+1}s_{i_r}u_{j_1\cdots j_{r-1}\hat{j_r}j_{r+1}\cdots j_i},\ i=2,\dots,n,
\] 
where $u_{j_1\cdots j_i}$ is a generator of $\Osh_X(-E_{j_1}-\cdots-E_{j_i})$ as $\Osh_X$-module.
\end{theorem}

\begin{proof}
It is enough to prove exactness at any local ring $R=\Osh_x$, $x\in X$. 
To simplify notation, we will still denote by $s_i$ its image in any local ring.
Given $x\in X$ we can assume that the image of $s_1$  in $\Osh_x$ is a unit 
since $\cap_{i=1}^n E_i=\emptyset$.
Let $E=R^n$ and let $\varphi_1:E\to R,\ \varphi(e_i)=s_i$. 
The sequence (\ref{koszul}) is the Koszul complex $K_{\cdot}(\varphi)$ \cite[Chapter XXI, p. 852]{SL}:
\[
\xymatrix{
0\ar[r] & \wedge^nE\ar[r]^{\varphi_n} & \cdots \ar[r]^{\varphi_3}& \wedge^2E\ar[r]^{\varphi_2} &E\ar[r]^{\varphi_1} & R\ar[r] & 0.
}
\]
We will denote by $H_pK(s_1,\dots,s_{n})$ the $p$-th homology group of the complex.
Observe that the group $H_0K(s_1,\dots,s_{n})$ is trivial since it is isomorphic to $R/(s_1,\dots,s_n)$ and $s_1$ is a unit.

We now prove that all homology groups with $p>0$ vanish by induction on $n$.
If $n=1$ the sequence is exact since $\varphi_1$ is the multiplication by $s_1$, which is a unit.
Now we assume exactness for $n-1$. By \cite[Theorem 4.5 a), Chapter XXI]{SL} there is an exact sequence 
of Koszul homology groups:
\[
\xymatrix@R=0.2cm{
H_pK(s_1,\dots,s_{n-1})\ar[r] & H_pK(s_1,\dots,s_{n-1})\ar[r] & H_pK(s_1,\dots,s_{n})\ar[r]&\\
\dots &\dots &\dots \\
H_1K(s_1,\dots,s_{n})\ar[r] & H_0K(s_1,\dots,s_{n-1})\ar[r] & H_0K(s_1,\dots,s_{n-1})\ar[r]&.
}
\]
For all $p>0$ the group $H_pK(s_1,\dots,s_{n})$ is between two groups which are zero by induction 
(or by the previous remark on $H_0K$), thus it is zero.
\end{proof}

Considering the case of two or three disjoint divisors, we obtain the following results.

\begin{corollary}\label{teokoszul2}
Let $X$ be a smooth projective variety over $\cc$, $E_1,E_2$ be effective divisors of $X$
such that $E_1\cap E_2=\emptyset$ and $s_{i}\in H^0(X, E_i)$ be a defining section for $E_i$, $i=1,2$. 
If $D\in {\rm WDiv}(X)$ is such that $h^1(X,D-E_1-E_2)=0$, then there is a surjective morphism
 \[
 H^0(X,D-E_1)\oplus H^0(X,D-E_2)\to H^0(X,D),\quad (f_1, f_2)\mapsto f_1s_1 + f_2s_2.
 \]
\end{corollary}

\begin{proof}
Consider the exact sequence of sheaves obtained tensoring sequence (\ref{koszul}) with $\Osh_X(D)$:
\[
\xymatrix{
0\ar[r] & \Osh_X(D-E_1-E_2)\ar[r] & \Osh_X(D-E_1)\oplus \Osh_X(D-E_2)\ar[r] & \Osh_X\ar[r] & 0.
}
\]
Taking the associated exact sequence in cohomology one obtains the statement.
\end{proof}

\begin{remark}\label{bea}
More generally there is an exact sequence of sheaves \cite[Lemma I.5]{AB}:
\[
\xymatrix@C=8pt{
0\ar[r] & \Osh_X(D-E_1-E_2)\ar[r] & \Osh_X(D-E_1)\oplus \Osh_X(D-E_2)\ar[r] & \Osh_X\ar[r] & \Osh_{E_1\cap E_2\ar[r]} &0.
}
\]
\end{remark}
 
 \begin{corollary}\label{teokoszul3}
 Let $X$ be a smooth projective variety over $\cc$, 
 $E_1,E_2,E_3$ be effective divisors of $X$
such that $E_1\cap E_2\cap E_3=\emptyset$ 
and $s_{i}\in H^0(X, E_i)$ be a defining section for $E_i$, $i=1,2,3$. 
If $D\in {\rm WDiv}(X)$ then the morphism 
\[
\bigoplus_{i=1}^3 H^0(X, D-E_i)\to H^0(X,D),\ (f_1,f_2,f_3)\mapsto f_1s_1+f_2s_2+f_3s_3,
\]
is surjective if one of the following occurs:
\begin{enumerate}
\item $h^1(X, D-E_i-E_j)=0$ for all distinct $i,j\in \{1,2,3\}$ and $h^2(X, D-E_1-E_2-E_3)=0$.
\item $h^1(X, D)=0$, $h^p(X, D-E_i-E_j)=0$ for $p=1,2$ and  for all distinct $i,j\in \{1,2,3\}$, and $h^2(X, D-E_1-E_2-E_3)=\sum_{i=1}^3h^1(X, D-E_i)$.
\end{enumerate}
\end{corollary}

 \begin{proof}
The exact sequence in Theorem \ref{koszul3} can be split into 
two short exact sequences
\begin{equation}\label{k1}
 0  \rightarrow \oo_X(D-E_1-E_2-E_3)\xrightarrow {d_3} \oplus_{i<j}\oo_X(D-E_i-E_j)\xrightarrow{d_2}  {\rm Im}(d_2)\rightarrow 0.
\end{equation}
and 
\begin{equation}\label{k2}
 0 	\rightarrow  {\rm Im}(d_2) \xrightarrow{i}  \oplus_{k=1}^3\oo_X(D-E_k)\xrightarrow {d_1} \oo_X(D)\rightarrow 0,
 \end{equation}
where $i$ is the inclusion morphism. 
These give rise to the following exact sequences in cohomology:
{\footnotesize
\[
  \oplus_{i<j}H^1(X, D-E_i-E_j)\rightarrow H^1(X, {\rm Im}(d_2))
  \rightarrow H^2(X, D-E_1-E_2-E_3)\rightarrow   \oplus_{i<j}H^2(X, D-E_i-E_j)
\]
\[
 \oplus_{k=1}^3H^0(X, D-E_k)\xrightarrow {\phi} H^0(X, D)
 \xrightarrow {\phi'} H^1(X, {\rm Im}(d_2))\xrightarrow{\phi''} \oplus_{k=1}^3H^1(X, D-E_k)\rightarrow H^1(X, D). 
\]
}
If  (i) holds, by the first sequence we obtain that $H^1(X, {\rm Im}(d_2))=0$, 
then by  the second sequence  the morphism $\phi$ is surjective.\\
On the other hand, if (ii) holds, since
$h^1(X, D)=0$ then $\phi''$ is surjective by the second sequence. 
Moreover by  the first sequence  we have that 
\[
{\rm dim}(H^1(X, {\rm Im}(d_2)))=h^2(X, D-E_1-E_2-E_3)={\rm dim}( \oplus_{k=1}^3H^1(X, D-E_k)),
\]
 then $\phi''$ is an isomorphism and one obtains the statement.
\end{proof}

\subsection{Computing Cox rings of K3 surfaces}\label{techniques K3}

We start proving a consequence of Corollary \ref{teokoszul3} for K3 surfaces.
We will denote by $\BNef(X)$ the Hilbert basis of the nef cone $\Nef(X)\subseteq \Cl(X)_\qq$.

\begin{lemma}\label{aux}
Let $X$ be a smooth projective K3 surface over $\cc$ 
and let $D= N_1+N_2+N_3$, where $N_1,N_2,N_3$ are nef and non zero.
Then there exist $N_1', N_2', N_3'$ nef, effective and non zero 
such that $D\sim N_1'+N_2'+N_3'$ and $N_1'\cap N_2'\cap N_3'=\emptyset$.
\end{lemma}

\begin{proof}
Observe that if $D$ is linearly equivalent to a sum of three nef non-zero 
divisors $D_i$ such that one of them is base point free, then 
up to linear equivalence we can assume that the $D_i$'s are effective divisors 
with $D_1\cap D_2\cap D_3=\emptyset$.
In particular, the result is proved if  $|N_i|$ is base point free for some $i=1,2,3$.

Assume now that $|N_i|$ has base points for all $i$.
By Proposition \ref{lsk3} ii) we have that $N_i\sim k_iF_i+E_i$, 
where $k_i\geq 2$ is an integer, $F_i$ is nef and primitive 
with $F_i^2=0$ and $E_i^2=-2, F_i\cdot E_i=1$.
If $E_i\cdot E_j\not=0$ for all $i,j$, then $E_1+E_2+E_3$ 
is nef and $D\sim k_1F_1+(k_2F_2+k_3F_3)+(E_1+E_2+E_3)$,
where $|k_1F_1|$ is base point free. 
Thus we conclude by the first remark in the proof.

Moreover, if $E_j\cdot F_i\not=0$ for some $i\not=j$
then $k_iF_i+E_i+E_j$ is nef and 
$D\sim (k_iF_i+E_i+E_j)+k_jF_j+(k_rF_r+E_r)$, 
where $i,j,r$ are distinct, and $|k_2F_2|$ is base point free.
As before, we conclude by the first remark in the proof.

Thus we can assume that $E_j\cdot F_i=0$ for all $i\not=j$ 
and that $E_1\cdot E_2=0$, up to permuting the $N_i$'s.
In this case  $N_1, N_2,N_3$ 
are linearly equivalent to three effective divisors with empty intersection
since $k_1F_1+E_1$ and $k_2F_2+E_2$ intersect at points in $F_1\cap F_2$ 
and $k_3F_3+E_3$ is linearly equivalent to an effective divisor not passing through such points,
since  $|k_3F_3|$ is base point free and $E_3\cap F_1$ is empty.
This concludes the proof.  \end{proof}

\begin{theorem}\label{gen}
Let $X$ be a smooth projective K3 surface over $\cc$ 
and let $D= \sum a_iN_i$, where $[N_i]$ are elements 
of $\BNef(X)$ and $a_i$ are positive integers.
If $R(X)$ has a generator in degree $[D]$, then either
$\sum a_i\leq 3$ or $D\sim 2(F+F')$ where $F,F'$ are distinct nef, primitive divisors 
with $F^2=(F')^2=0$ and $F\cdot F'=2$.
\end{theorem}

\begin{proof}
Assume that $D= \sum a_iN_i$ with $\sum_ia_i\geq 4$, $a_i> 0$ and $[N_i]\in \BNef(X)$.
By the hypothesis on $D$  we can find three nef divisors
$N_1, N_2, N_3$  such that $D-\sum_{i=1}^3N_i$ is nef and non zero.
Moreover, up to linear equivalence, we can assume that the $N_i$'s are effective 
divisors with $N_1\cap N_2\cap N_3=\emptyset$ by Lemma \ref{aux}.
 
The divisors $A_{ij}:=D-N_i-N_j$, with distinct $i,j\in \{1,2,3\}$, are nef, 
thus by Proposition \ref{lsk3} $h^1(X, A_{ij})=0$ unless $A_{ij}\sim kF$, 
where $F$ is nef and primitive with $F^2=0$ and $k\geq 2$ is an integer.
Moreover $h^2(X, D-N_1-N_2-N_3)=h^0(X, N_1+N_2+N_3-D)=0$ since $D-N_1-N_2-N_3$ is an effective non zero divisor.
Thus, unless $A_{ij}\sim kF$, we conclude by Corollary \ref{teokoszul3}.

We now consider the case $A_{ij}\sim kF$, that is $D\sim N_i+N_j+kF$ with $k\geq 2$ and $F$ as above. 
We have that $h^1(X, D-2F)=0$ unless $D\sim 2F+\ell F'$, where $F'$ 
is nef and primitive with $(F')^2=0$ and $\ell\geq 2$.
This case is considered in Lemma \ref{ss'}, which shows 
that $R(X)$ is not generated in degree $[D]$ unless 
$D\sim 2(F+F')$ with $F\cdot F'=2$.

Assuming that $D\not\sim 2F+\ell F'$ where $F,F'$ are as above, 
we now prove that either $h^1(X, D-F-N_i)$ or $h^1(X, D-F-N_j)$ is zero.
Assume on the contrary that these are both non zero. 
By Proposition \ref{lsk3} ii),
\[
D-F-N_i\sim N_j+(k-1)F\sim \ell_i F_{i},
\]
where $F_i$ is nef and primitive with $(F_i)^2=0$ and $\ell_i\geq 2$. 
This implies that $F\sim F_i$ and $N_j\sim (\ell_i-k+1)F$.
The same argument for $D-F-N_j$ gives that $N_i\sim (\ell_j-k+1)F$.
Thus $D\sim rF$ with $r\geq 4$. 
By Lemma \ref{ss'} in this case $R(X)$ is not generated in degree $[D]$.

Thus we can assume that $h^1(X, D-F-N_i)=0$.
Taking $E_1,E_2\in |F|$  distinct and $E_3\in |N_i|$
 we can thus conclude applying Corollary \ref{teokoszul3} i).
Observe that $h^2(X,D-E_1-E_2-E_3 )=h^2(N_j+(k-2)F)=0$ since $N_j+(k-2)F$ is an effective non zero divisor.
 \end{proof}
 
\begin{lemma}\label{ss'}
Let $X$ be a K3 surface and let $D=aF+bF'$, 
where $F,F'$ are distinct primitive nef divisors on $X$ with $F^2=(F')^2=0$ 
and $a,b$ are integers with $a\geq b\geq 0$.
Then $R(X)$ has no generators in degree $[D]$ if one of the following holds:
\begin{enumerate}
\item $b=0$ and $a\geq 2$; 
\item $a\geq 3$;
\item $a=b=2$ and $F\cdot F'>2$.
\end{enumerate}
\end{lemma}

\begin{proof}
If $b=0$, that is $D=aF$, then $R(X)$ is not generated in degree $[D]$ since 
$H^0(X, aF)$ is the $a$-th symmetric power of $H^0(X, F)$.

Now assume $b>0$ and $a\geq 3$.
We now prove that $h^1(X, D-2F)=0$.
If $h^1(X, D-2F)$ is not zero, then by Proposition \ref{lsk3} ii)
\[
D- 2F\sim (a-2)F+bF'\sim rF'',
\]
where $r\geq 2$ is an integer and $F''$ is nef and primitive with $(F'')^{2}=0$ 
The previous relation implies that $((a-2)F+bF')^2$ and thus, since $a>2$,
$F\cdot F'=0$. Since $F,F'$ are fibers of elliptic fibrations, 
this means that $F\sim F'$, contradicting our hypothesis.
By Corollary \ref{teokoszul2} with $E_1,E_2\in |F|$ distinct  
we conclude that $R(X)$ has no generators in degree $[D]$.

We finally consider the case $a=b=2$, that is when $D=2(F+F')$.
Observe that the fibers $F,F'$ of two distinct elliptic fibrations 
can not have intersection number one, since otherwise 
$F$ would be mapped isomorphically to $\pp^1$ by the morphism associated to $|F'|$.
Thus $(F+F')^2=2F\cdot F'>2$. Moreover, if $E$ is any elliptic curve,
then by the previous remark $(F+F')\cdot E> 2$ unless $E=F$, or $E=F'$, and $F\cdot F'=2$.
If $F\cdot F'>2$, then $F+F'$ is not hyperelliptic by Proposition \ref{hyp}.   
By Proposition \ref{ray} $R(X)_{[F+F']}$ is generated in degree one, in particular 
$R(X)$ is not generated in degree $[D]$.
\end{proof}

The following result  shows that $R(X)$ is not generated in degrees which are 
sums of a very ample class and the class of an elliptic fibration under certain conditions.

 \begin{lemma}\label{va}
Let $X$ be a K3 surface and $D=F+D'$ a nef divisor, where $F$ is nef with $F^2=0$ 
and $D'$ is very ample. Assume that $F\sim E_1+E_2$, where $E_1,E_2$ are $(-2)$-curves and 
that  the image of the natural map
\[
\phi: H^0(D-E_1)\oplus H^0(D-E_2)\to H^0(D)
\]
has codimension two. Then $R(X)$ has no generator in degree $[D]$.
\end{lemma}

\begin{proof}
Observe that $F$ defines an elliptic fibration $\varphi_{|F|}:X\to \pp^1$ 
and $E_1,E_2$ are the components of a reducible fiber of  $\varphi_{|F|}$.
Thus $E_1,E_2$ intersect at two points $p,q$,  which could be infinitely near.
Let $V_{p,q}\subset H^0(D)$ be the subspace of sections vanishing at $p$ and $q$.
Since the image of $\phi$ has codimension two, then it coincides with $V_{p,q}$.
Since $D'$ is very ample there are two sections $s_1,s_2\in H^0(D')$ such that 
$s_1(p)$ and $s_1(q)$ are not zero  and such that $s_2(p)=0$ and $s_2(q)\not=0$.  
Let $t\in H^0(F)$ be a section not vanishing on $E_1+E_2$.
The sections $s_1t$ and $s_2t$, together with $V_{p,q}$, generate $H^0(D)$.
 \end{proof}

We conclude this section recalling a result by Ottem  \cite[Proposition 2.2]{JCO}.

\begin{proposition}\label{ottem}
Let $X$ be a smooth projective K3 surface. Let $A$ and $B$ be nef divisors on $X$ such that $|B|$ is base point free. Then the multiplication map
\[
H^0(X,A)\otimes H^0(X,B)\to H^0(X,A+B)
\]
is surjective if $H^1(X,A-B)=H^1(X,A)=0$ and $H^2(X,A-2B)=0$.
 \end{proposition}

\section{K3 surfaces of Picard number three}
By Theorem \ref{classification} there are $26$ families of K3 
surfaces with Picard number three whose general member has finitely generated Cox ring. 
These families have been identified and studied by V.V. Nikulin in \cite{VN}. 
In this section we will determine the degrees of a generating set of $R(X)$ for each such family.

\subsection{Effective cones}\label{effective cone}
By Corollary \ref{gendegree} the Cox ring has a 
generator in degree $[D]$ for each $[D]$ in the Hilbert basis of ${\rm Eff}(X)$.
Moreover, by Theorem \ref{gen}, the Hilbert basis of the nef cone ${\rm Nef}(X)$ 
also has a key role in the computation of $R(X)$.
We computed such bases by means of a computer program written in Magma, 
see section \ref{magma}.
Observe that in \cite{VN} the author already computed the set of $(-2)$-curves of each family.

\begin{proposition}\label{eff-nef}
Table \ref{eff&nef}  describes the extremal rays and the Hilbert bases of ${\rm Eff}(X)$ and ${\rm Nef}(X)$ 
for each of the $26$ families of Mori dream K3 surfaces of Picard number three.
\end{proposition}

In the tables and in the proof of the following theorem we will adopt this notation:

\begin{itemize}
\item $\Eff(X)$ is the effective cone of $X$, $E(X)$ the list of its extremal rays (i.e. the classes of  $(-2)$-curves) and $\BEff(X)$ its Hilbert basis,
\item $\Nef(X)$ is the nef cone of $X$, $N(X)$ the list of extremal rays of $\Nef(X)$ and $\BNef(X)$ its Hilbert basis.
\end{itemize}

 \subsection{Generators of $R(X)$}\label{sec-gen}

\begin{theorem}\label{main}
Let $X$ be a  Mori dream K3 surface of Picard number three. 
The degrees of a set of generators of the Cox ring $R(X)$ are given in
Table \ref{TableGen}. All degrees in the Table are necessary to generate 
$R(X)$, except eventually for those marked with a star.
\end{theorem}

\begin{proof}
 By Corollary \ref{gendegree} the Cox ring has a generator in degree $[D]$ for each $[D]\in {\rm BEff}(X)$.
Moreover, if $[D]$ is not nef, there exists a $(-2)$-curve $C$ such that $D\cdot C<0$, 
so that $C$ is  contained in the base locus of $|D|$ and the multiplication map $H^0(X,D-C)\rightarrow H^0(X,D)$ 
by a non-zero element of $H^0(X,C)$ is surjective. Thus we can assume $[D]$ to be nef.

By Proposition \ref{lsk3} the linear system of any nef divisor is base point free unless there exists 
a smooth elliptic curve $F$ and a $(-2)$-curve $E$ such that $E\cdot F=1$.
This happens only for the case $21$, i.e. the family of K3 surfaces with 
Picard lattice isometric to $S_{1,1,1}\cong U\oplus A_1$, 
whose Cox ring has been computed in \cite[Proposition 6.6, ii)]{A.H.L}.
Moreover, the only case where there exist two smooth elliptic curves $D,D'$ with $D\cdot D'=2$ 
is case $22$, i.e. the family of K3 surfaces with Picard lattice isometric to $S_{1,1,2}\cong U(2)\oplus A_1$, 
whose Cox ring has been computed in \cite[Proposition 6.6, i)]{A.H.L}. 
For the following arguments we exclude the cases $21$ and $22$.

By Theorem \ref{gen} it is enough to consider those nef degrees which are sums of at most 
three elements of the Hilbert basis of the nef cone.
This allows to form a finite list of possible nef degrees, which is then analysed using the 
techniques in section \ref{techniques} and with the help of a Magma program 
described in section \ref{magma}.
More precisely, these are the main steps. Let $L$ 
be the set of degrees which are sums of at most 
three elements in ${\rm BNef}(X)$ 
and consider the following three sets:
\[
\begin{array}{l}
T1:=\{\{A,B\}: A,B\in E(X)\cup {\rm BNef}(X),\  A\cdot B=0\},\\[5pt]
T2:=\{\{E_1,E_2,E_3\}: E_i\in E(X)\cup {\rm BNef}(X),\  E_3\not\in E(X)\},\\[5pt]
T3:=\{(A,B): A,B\in {\rm BNef}(X), h^1(A-B)=h^1(A)=h^0(2B-A)=0\},\\[5pt]
T4:=\{3A: A\in {\rm BNef}(X),\ A^2\not=2\}\cup \{2A: A\in {\rm BNef}(X), A \text{ is not hyperelliptic or } A^2=2\}.
\end{array}
\]

We apply the following tests to any element $D\in L$:
\begin{enumerate}[Test 1.]
\item  checks whether $h^1(X,D-A-B)=0$ for some $\{A,B\}\in T1$. 
If this holds, then $R(X)$ has no generator in degree $[D]$ by  Corollary \ref{teokoszul2} and Test 1 returns false.
\item  checks  whether there exists $\{E_1,E_2,E_3\}\in T2$ 
such that $h^1(X,D-E_i-E_j)=0$ for all $i,j\in \{1,2,3\}$ and $h^2(X,D-E_1-E_2-E_3)=0$. 
If this holds, then  $R(X)$ has no generator in degree $[D]$ by  Corollary \ref{teokoszul3} and Test 2 returns false.
\item   checks whether $D$ can be written as a sum $A+B$, where $(A,B)\in T3$. 
If this holds, then $R(X)$ has no generator in degree $[D]$ by  Proposition \ref{ottem} and Test 3 returns false.
\item checks whether $[D]\in T4$. 
If this holds, then  $R(X)$ has no generator in degree $[D]$ by  Proposition \ref{ray}. 
and Test 4 returns false.
\item  if $D$ is a sum of two elements of ${\rm BNef}(X)$, it checks whether $D$ satisfies the hypotheses of Lemma \ref{va}.
If this holds, then $R(X)$ has no generator in degree $[D]$ and Test 5 returns false.
\item if $D$ is a sum of three elements of ${\rm BNef}(X)$, it checks the same property of Test 3, 
where $A$ is a sum of two elements in 
 ${\rm BNef}(X)$ and $B\in {\rm BNef}(X)$. 
 If this holds, then $R(X)$ has no generator in degree $[D]$ by Proposition \ref{ottem} 
 and Test 6 returns false.
\end{enumerate}
Let $G$ be the set containing the degrees of all $(-2)$-curves and the degrees in $L$ 
for which the tests are true. In order to determine which such degrees are 
necessary to generate $R(X)$, we apply the function \texttt{Minimal}. 
Given $[D]\in G$, the function 
first finds all possible ways to write $[D]$ as a linear combination with non-negative coefficients 
of elements in $G\backslash\{[D]\}$. If any such linear combination contains 
the degree of a $(-2)$-curve $E$, or one of the degrees of two $(-2)$-curves $E_1,E_2$
with positive intersection, then all sections in $H^0(G)$ which are polynomials 
in sections of $\cup_{A\in G\backslash\{[D]\}}H^0(A)$ have either $E$ or $E_1\cap E_2$ 
in the base locus. Since $|D|$ is base point free, this implies that $[D]$ is necessary 
to generate $R(X)$.
Moreover, if any linear combination as before contains one among the degrees 
of three $(-2)$-curves $E_1,E_2,E_3$ and $D$ satisifies the conditions 
of Lemma \ref{l2} with such curves, then $[D]$ is necessary 
to generate $R(X)$.

Finally, we apply Lemma \ref{l1} to show that generators in certain degrees of 
type $[3D]$ are not necessary.
\end{proof}

 \begin{lemma}\label{l2}
 Let $X$ be a K3 surface and let $D=E_1+E_2+E_3$ be a base point free divisor, 
 where $E_1,E_2,E_3$  are $(-2)$-curves such that $h^1(E_i+E_j)=0$ 
 for all distinct $i,j$. Then the natural map
 \[
 \psi:\bigoplus_{i=1,2,3} H^0(D-E_i)\to H^0(D)
 \]
 is not surjective. Moreover, if $E_1,E_2,E_3$  are disjoint, the image of $\psi$ 
 has codimension one.
 \end{lemma}
 
 \begin{proof}
 If $E_1\cap E_2\cap E_3$ is not empy, then $\psi$ 
 is clearly not surjective, since $D$ is base point free.
On the other hand, if $E_1\cap E_2\cap E_3$ is empty,
then we can consider the associated Koszul exact sequence 
of sheaves in Theorem \ref{koszul3}, which gives rise 
to the two short exact sequences \eqref{k1} and \eqref{k2}.
The first sequence, using the fact that $h^0(E_i)=1, h^1(E_i)=h^2(E_i)=0$  
for $i=1,2,3$ and $h^1(\Osh_X)=0$, $h^2(\Osh_X)=1$, gives 
$h^0({\rm Im}(d_2))=2$ and $h^1({\rm Im}(d_2))=1$.
Using this and the fact that $h^1(E_i+E_j)=0$ in the second sequence, 
we find that the image of $\psi$ has codimension one in $H^0(D)$.
 \end{proof}
 
 \begin{lemma}\label{l1} Let $X$ be a K3 surface, $D$ be a nef and base point free divisor with $D^2=2$ 
and $i$ be the covering involution of the associated double cover $\varphi_{|D|}$. 
The Cox ring $R(X)$ has no generator in degree $[3D]$ if there exists a $(-2)$-curve 
$E$ which is not invariant for $i$ and such that $3D-E$ is effective and base point free.
\end{lemma}

\begin{proof}
Observe that ${\rm Sym}^3H^0(D)$ is a codimension one subspace of $H^0(3D)$ 
and is the invariant subspace for the action of $i^*$ on $H^0(3D)$.
Since $3D-E$ is effective and base point free, then there exists a non-constant 
section in $H^0(3D)$ of the form $ss_{E}$, where $s_E\in H^0(E)$ and $s\in H^0(3D-E)$ 
is not divisible by $i^*(s_E)$. Such section is not $i$-invariant, thus it generates 
 $H^0(3D)$ together with ${\rm Sym}^3H^0(D)$.
\end{proof}

\subsection{Magma programs} \label{magma}

In this section we briefly present the Magma programs used for the proof of Theorem \ref{main}.
We include as ancillary files in the arXiv version of the paper, the following files:

\begin{enumerate}[$\bullet$]
\item \texttt{LSK3Lib.m}: library for linear systems on K3 surfaces
\item\texttt{Find-2.m}:  library for computing the set of $(-2)$-curves of a Mori dream K3 surface\item \texttt{TestLib.m}: library containing the test functions described in the proof of Theorem \ref{main}
\item \texttt{MinimalLib.m}:  library containing functions which check the minimality of a generating set of $R(X)$
\item \texttt{K3Rank3.txt}: text file containing the intersection matrix and the list of classes of $(-2)$-curves for all Mori dream K3 surfaces of Picard number $3$
\item \texttt{Gen(K3Rank3)}: text file containing the list of classes which pass all tests 
in  \texttt{TestLib.m} (thus contains the degrees of a generating set of $R(X)$), for all Mori dream K3 surfaces of Picard number $3$
\end{enumerate}

\noindent We briefly describe the functions contained in each of the libraries.\\ 

\centerline{{\bf \texttt{LSK3Lib.m}.}}

\begin{enumerate}[$i.$]
\item \texttt{qua}: returns the intersection product of two vectors given the intersection matrix
\item \texttt{h01}, \texttt{h0}, \texttt{h1}: compute $h^0$ and $h^1$ of a divisor on a K3 surface given the set of classes of $(-2)$-curves and the intersection matrix
\item \texttt{Eff} and \texttt{HBEff}: compute the effective cone and a Hilbert basis of it
\item \texttt{Nef} and \texttt{HBNef}: compute the nef cone and a Hilbert basis of it
\item \texttt{Hyperelliptic}: checks whether a divisor on a K3 surface is hyperelliptic  
\item \texttt{IsNef}: checks whether a divisor on a K3 surface is nef  
\item \texttt{IsVAmple}: checks whether a divisor on a K3 surface is very ample\\
\end{enumerate}

\centerline{{\bf \texttt{Find-2.m}.}}

\begin{enumerate}[$i.$]
\item \texttt{qua}: same as before 
\item \texttt{Pts}:  given a diagonal matrix $D\in {\rm GL}(n,\qq)$ with $D_{1,1}>0$ and $D_{i,i}<0$ for $i\not=1$, 
$B \in {\rm GL}(n,\qq)$ and a non-negative integer $m$ returns 
the list \texttt{Pts(D,B,m)} of vectors $y\in \qq^n$ such that $y_1=m, y^TDy=-2$. 
\item \texttt{Test}: given an intersection matrix and a list of vectors, 
computes the cone $C$ generated by the vectors and 
returns true if, for any facet  $F$ of $C$, the intersection matrix of  
the vectors generating  $F$ is negative semidefinite.
\item \texttt{FindEff}: given an intersection matrix, 
returns the set of  classes of $(-2)$-curves 
with respect to a choice of an ample class. 
More precisely it follows these steps:
\begin{enumerate}
\item it finds a diagonal matrix $D\in {\rm GL}(n,\qq)$ 
as in function \texttt{Pts} and $B \in {\rm GL}(n,\qq)$ such that $D=BQB^T$;
\item  computing \texttt{Pts(D,B,0)}, it finds the list $L$ of all vectors $v=B^Ty\in \zz^n$ 
such that $v^2=-2$ and $y_1=0$ (this gives a root system);
\item if $\#L\geq 2$,   
after choosing randomly an integral combination $H$ of  the vectors in $L$
having non zero intersection with all of them, it finds the list  $L_0$ of simple roots 
having positive intersection with $H$;
\item  for $i\geq 1$ it finds inductively the list $L_i$ of all vectors $v=B^Ty\in \zz^n$ such that $v^2=-2$ and 
 $y_1=\frac{i}{d}$ with $d=|{\rm Det}(B)|$ having non-negative intersection with all vectors in $L_{i-1}$ 
 and at each step it defines $L$ as the union of $L_{i-1}$ with $L_i$;
 \item when the function \texttt{Test} applied to the list of vectors $L$ 
 returns true, then the program returns  $L$.
\end{enumerate}
\end{enumerate}

In the following functions, $D$ will denote the class of a divisor,
\texttt{hb} the list of vectors in the Hilbert basis of the nef cone and
\texttt{Q} an intersection matrix.\\

\centerline{\bf \texttt{TestLib.m}} 

\begin{enumerate}[$i.$]
\item \texttt{S}: given $D$,  \texttt{neg}, \texttt{hb} and \texttt{Q}, 
it gives the set of all classes $w$ in either $\texttt{neg}$ or $\texttt{hb}$ which are distinct from $D$ and such that $D-w$ is effective.
\item \texttt{Ti} (i=1,2,3,4) and \texttt{testi} (i=1,2,3,4,5,6): constructs the set $Ti$ and performs the Test i described in the proof of Theorem \ref{main}). 
\item \texttt{test}:  it checks if a divisor passes  test1, test2, test3, and test4.
\item \texttt{gen}: given \texttt{neg} and \texttt{Q}, it returns three lists $L_1, L_2, L_3$ 
of classes of divisors: $L_1$ contains the classes in \texttt{hb} for which \texttt{test} is true,
 $L_2$  the sums of two elements in \texttt{hb} for which  \texttt{test} and \texttt{test5} are true
and $L_3$ the sums of three elements in \texttt{hb} for which  \texttt{test} and \texttt{test6} are true.\\
\end{enumerate}

\centerline{\bf \texttt{MinimalLib.m}}
\begin{enumerate}[$i.$]
\item \texttt{RR}: given a matrix $M\in M_{m\times n}(\zz)$ and a vector $v\in \zz^m$, 
it finds all $w\in \zz^n$ with non-negative integral coefficients such that  $Mw=v$.
\item \texttt{SG}:  given  $D$,   \texttt{neg} 
  \texttt{hb}, $\texttt{Q}$ and a set of classes of divisors $G$, it finds all classes  $w$ in $G$  such that 
$D-w$ is effective.
\item \texttt{RRD}: it finds all possible ways to write  $D$ as a linear combination 
with non-negative integer coefficients of the classes  in \texttt{SG(D,neg,hb,Q,G)}.
\item \texttt{Minimal}: given $D$,  \texttt{neg}, \texttt{hb}, $\texttt{Q}$ 
and a set of classes of divisors $G$, it first computes \texttt{SG(D,neg,hb,Q,G)}.
If the latter is empty, then it returns true.
Otherwise, it finds all possible ways of writing $D$ 
as a non-negative linear combination 
of elements in $G$, using the function \texttt{RR}.  
After this, it computes three lists: the list  $B1$ gives the classes of $(-2)$-curves 
which appear in any possible writing of $D$ as before, the list $B2$ 
gives the pairs $E_1,E_2$ of classes of $(-2)$-curves such that 
either $E_1$ or $E_2$ appears in any writing of $D$  and $E_1\cdot E_2>0$,
the list $B3$ gives the triples $E_1,E_2,E_3$ of classes of $(-2)$-curves 
such that one of them appears in any writing of $D$ and such that 
the hypotheses of Lemma \ref{l2} are satisfied. 
If one among $B1$, $B2$ and $B3$ is not empty, then the function returns true 
and the three lists.
\end{enumerate}

\subsection{Some special cases}\label{special}

Looking at Table \ref{TableGen} one can see that there are two cases 
where $R(X)$ is generated in six degrees  
(these are $S_{1,1,1}$ and $S_{1,1,2}$, see Remark \ref{double}) 
and other cases where it is generated in  seven degrees:
$S_1, S_5, S_{4,1,1}, S_{1,3,1}, S_{1,1,3}$ and $S_{1,1,4}$.
In this section we will provide a presentation for the Cox ring 
of a very general member of the families $S_1$ and $S_{4,1,1}$.
We expect that similar techniques can provide a presentation 
of $R(X)$ also in the remaining cases.

 \begin{example}[Case $S_1$]\label{case1}
Let $X$ be a K3 surface with $\Cl(X)\cong S_1=(6)\oplus 2A_1$.
We denote the natural generators of $\Cl(X)$ by $e_1,e_2,e_3$ with
\[
e_1^2=6,\ e_2^2=e_3^2=-2,\ e_i\cdot e_j=0\, \text{ for } i\not=j.
\]
By Proposition \ref{eff-nef} the classes of the $(-2)$-curves can be taken to be:
\[
\begin{array}{lll}
f_1=e_2 & f_2=2e_1-3e_2-2e_3\\
f_3=e_3 & f_4= 2e_1-2e_2-3e_3\\
f_5=e_1-2e_2 & f_6=e_1-2e_3.
\end{array}
 \]

The Hilbert basis of the effective cone contains, besides the previous classes,  
the ample class $h:=e_1-e_2-e_3$. 
We now determine a presentation for $R(X)$, in particular we show that 
it is a complete intersection.

\begin{theorem}\label{S1}
Let $X$ be a K3 surface with $\Cl(X)\cong S_1=(6)\oplus A_1^2$. Then $X$ 
can be defined by an equation of the following form in $\pp(1,1,1,3)$:
\[
w^2=q_1(x_0,x_1,x_2)q_2(x_0,x_1,x_2)q_3(x_0,x_1,x_2)+f(x_0,x_1,x_2)^2,
\]
where the $q_1,q_2,q_3$ are homogeneous of degree $2$ and $f$ is homogeneous 
of degree $3$.
Moreover,  the Cox ring of a very general $X$ as before is isomorphic to 
$\cc[s_1,\dots, s_6,x_0,x_1,x_2]/I$, where the degrees of the generators are given by the columns of the following matrix

\[
\left(
\begin{array}{ccccccccc}
0 & 2 & 0 & 2 & 1 & 1 & 1 & 1 & 1\\
1 & -3 & 0 & -2  & -2 & 0 & -1 & -1 & -1\\
0 & -2 & 1 & -3 & 0 & -2 & -1 & -1 & -1
\end{array}
\right)
\]

and the ideal $I$ is generated by the following polynomials:
\[ 
s_1s_4s_5+s_2s_3s_6-f(x_0,x_1,x_2).
\]
\[
s_1s_2-q_1(x_0,x_1,x_2), s_3s_4-q_2(x_0,x_1,x_2), s_5s_6-q_3(x_0,x_1,x_2).
\]
\end{theorem}

\begin{proof} Let $h=e_1-e_2-e_3\in \Cl(X)$, with the previous notation. 
Observe that  $h^2=2$, $h$ is ample and the associated linear system is base point free 
by Lemma \ref{lsk3}, thus it defines a degree two covering $\pi:X\to \pp^2$ ramified along a smooth 
sextic curve $B=\{F(x_0,x_1,x_2)=0\}$. Since $2h=f_1+f_2=f_3+f_4=f_5+f_6,$ 
the image by $\pi$ of the six $(-2)$-curves of $X$ are three smooth conics $Q_1,Q_2,Q_3\subseteq\pp^2$ 
such that $\pi^{-1}(Q_i)$ is the union of two smooth rational curves for each $i=1,2,3$.
By looking at the intersection graph of the $(-2)$-curves one can see that $\pi_{|\pi^{-1}(D)}\to D$ is trivial,  
thus by Lemma \ref{rmk1}  there exists a plane cubic $C=\{f(x_0,x_1,x_2)=0\}\subset \pp^2$ 
such that $B\cdot D=2C\cdot D$.
Let $\bar\lambda_0,\bar\lambda_1\in \cc$ such that the curve
\[
G_{\lambda_0,\lambda_1}(x_0,x_1,x_2):=\lambda_0q_1(x_0,x_1,x_2)q_2(x_0,x_1,x_2)q_3(x_0,x_1,x_2)+\lambda_1f(x_0,x_1,x_2)^2=0,
\]
intersects $B$ in $19$ distinct points, i.e. in $B\cap D$ and one more point.
By Bezout's theorem, since $B$ and $\{G_{\bar\lambda_0,\bar\lambda_1}(x_0,x_1,x_2)=0\}$ 
intersect in at least $37$ points counting multiplicity
and $B$ is irreducibile, then $F(x,y,z)=\alpha G_{\bar\lambda_0,\bar\lambda_1}(x_0,x_1,x_2)$ for some $\alpha \in \cc^{*}$. 
This proves the first part of the statement.

The degrees of the generators of $R(X)$ are given in Theorem \ref{main}.
Clearly any minimal generating set of $R(X)$ must contain the sections $s_1,\dots,s_6$ defining 
the $(-2)$-curves of $X$ and a basis $x_0,x_1,x_2$ of $H^0(h)$.
The first three relations are obvious, due to the fact that $s_1s_{2}$ defines the preimage 
of the conic $Q_1$, similarly for the other two cases.
Observe that $h^0(3h)=11$, $h^0({\rm Sym}^3H^0(h))=10$ and $s_1s_4s_5, s_2s_3s_6\in H^0(3h)$.
Moreover, ${\rm Sym}^3H^0(h)$ is the invariant subspace of $H^0(3h)$ for the natural action of the covering involution $i$ of $\pi$.
Since $s_1s_4s_5+s_2s_3s_6$ is invariant, then it belongs to ${\rm Sym}^3H^0(h)$. This gives the last relation.

We will now prove that $I$ is prime. 
Let $g_1,g_2,g_3,g_4$ be the generators of $I$ (in the order given in the statement), 
let $X_0=\cc^9$, $X_i=V(g_1,\dots,g_i)\subset \cc^9$ for $i=1,2,3$
 and $L_i$ be the linear system on $X_i$ 
generated by the divisors cut out by the monomials of $g_{i+1}$, for $i=0,1,2,3$.
The key remark is that, by the generality assumptions on $f$ and $q_1,q_2,q_3$, 
the zero set of $g_{i+1}$ is the general element of the linear system $L_i$ (up to a coordinate change in the variables $s_i$ for $g_1$).
The linear system $L_i$ has no components in its base locus and is not composed with a pencil 
for each $i=1,2,3$, since it can be easily checked that its subsystem generated by the 
monomials of $g_{i+1}$ in $x_0,x_1,x_2$ already satisfies both properties.
It follows that $X_i$ is irreducible by Bertini's first theorem \cite[Theorem 3.3.1]{Lazarsfeld},
i.e. $I$ is prime.
Since then the ring $\cc[s_1,\dots,s_9,x_0,x_1,x_2]/I$ is an integral domain, has Krull dimension $5$ and   surjects onto $R(X)$, then it is 
 isomorphic to  $R(X)$.
\end{proof}

The following is well-known, see for example  \cite[Proposition 1.7, Ch.3]{Vermeulen}.

\begin{lemma}\label{rmk1}
Let $B\subset \pp^2$ be a smooth plane curve of degree $6$, let $\pi:X\to \pp^2$ be the $2:1$ cover of $\pp^2$ branched along $B$, and let $D\subset \pp^2$ be a curve not containing components of $B$. The restriction of the cover:
$$\pi|_{\pi^{-1}(D)}:\pi^{-1}(D)\to D$$
is trivial if only if there exists a curve $C\subset \pp^2$ of degree $3$ such that $B\cdot D=2C\cdot D$.
\end{lemma}
\end{example}

\begin{example}[Case $S_{4,1,1}$]\label{S411}
Let $X$ be a K3 surface with $\Cl(X)\cong S_{4,1,1}$, whose intersection matrix is
\[
\left[
 \begin{array}{rrr}
  -32 & 0 & 4\\
   0 & -2 & 2\\
  4 & 2 & -2
 \end{array}
 \right].
\]
By Proposition \ref{eff-nef} the classes of the $(-2)$-curves can be taken to be:
\[
f_1=e_2,\quad f_2=e_3,\quad f_3=e_1+3e_2+4e_3.
\]
The Hilbert basis of the nef cone contains the ample class $h:=e_1+4e_2+5e_3$, 
with $h^2=6$ and $h\cdot f_i=2$ for $i=1,2,3$, and three classes of elliptic fibrations 
$h_1=e_2+e_3$, $h_2=e_1+3e_2+5e_3$ and $h_3=e_1+4e_2+4e_3$.

\begin{theorem}\label{S411}
Let $X$ be a K3 surface with $\Cl(X)\cong S_{4,1,1}$.  
Then $X$ can be defined as the zero set in $\pp^4$ of two equations of the following form 
\[
 f_2(x_0,\dots, x_4)=0,\ \ell_1\ell_2\ell_3(x_0,\dots,x_3)+x_4g_2(x_0,\dots,x_4)=0,
\]
where  $f_2,g_2$ are homogeneous of degree $2$ and $\ell_1,\ell_2,\ell_3$ are homogeneous 
of degree $1$.
The Cox ring of a very general $X$ as before is isomorphic to 
$\cc[s_1,s_2,s_3,t_1,t_2,t_3,t]/I$, where the degrees of the generators are given by the columns of the matrix
\[
\left(
\begin{array}{ccccccccc}
0 & 0 & 1 & 0 & 1 & 1& 1\\
1 & 0 & 3 & 1 & 3 & 4& 4\\
0 & 1 & 4 & 1 & 5 & 4& 5
\end{array}
\right)
\]
and the ideal $I$ is generated by the polynomials:
\[
 f_2(s_1s_2s_3,s_1t_2,s_2t_3,s_3t_1,t),\ t_1t_2t_3+g_2(s_1s_2s_3,s_1t_2,s_2t_3,s_3t_1,t).
\]
\end{theorem}

\begin{proof}
Observe that the class $h$ is ample and not hyperelliptic,
thus it defines an embedding of $X$ in $\pp^4$ as complete intersection of a quadric and a cubic hypersurface.
Observe that $h=f_1+f_2+f_3$. 
\[
h=f_1+h_2=f_2+h_3=f_3+h_1=f_1+f_2+f_3,\quad 2h=h_1+h_2+h_3.
\]
This means that $X$ has three reducible hyperplane sections, which are union of 
a conic and an elliptic curve of degree $4$. 
The three conics are contained in a hyperplane $H$ 
and the three elliptic curves are contained in a quadric $K$.
Up to a coordinate change we can assume that $H=\{x_4=0\}$. 
Each conic is contained in a plane, 
thus $X\cap H=\{x_4=h_2=\ell_1\ell_2\ell_3=0\}$, where $h_2, \ell_i\in \cc[x_0,x_1,x_2,x_3]$ 
are homogeneous, $\deg(h_2)=2$ and $\deg(\ell_i)=1$. 
This implies that $X$ has an equation as in the statement
with $f_2=h_2\, ({\rm mod }\ x_4)$. Observe that the equation of the quadric $K$ 
is $g_2=0$.
By Theorem \ref{main} the Cox ring $R(X)$ is generated in the following degrees: 
\[
f_1,\ f_2,\ f_3,\ h,\ h_1,\ h_2,\ h_3. 
\]
Let $s_i$ be a generator of $H^0(f_i)$ for $i=1,2,3$. 
Moreover let  $s_1s_2, t_1$ be a basis of $H^0(h_1)$,
 $s_2s_3, t_2$ be a basis of $H^0(h_2)$
and $s_1s_3, t_3$ be a basis of $H^0(h_3)$.
Observe that $h^0(h)=5$ and $H^0(h)$ contains the 
subspace $S$ generated by $s_1s_2s_3,s_1t_2,s_2t_3,s_3t_1$.
Given a linear combination of such sections with coefficients 
$\alpha_1,\dots,\alpha_4\in \cc$ and evaluating it 
at a point $p$ where $s_1(p)=s_2(p)=0$ gives $\alpha_4s_3t_1(p)=0$.
By the generality assumption the three conics $s_i=0$ do not have a common intersection,
thus $s_3(p)\not=0$. Moreover $t_1(p)\not=0$ since $t_1$ defines a section 
of the elliptic fibration associated to $h_1$ which is distinct from $s_1s_2=0$. Thus $\alpha_4=0$.
The same  argument  for the pairs $s_1,s_3$ and $s_2,s_3$ gives $\alpha_2=\alpha_3=0$.
Thus $\dim(S)=4$.  
This implies that a  set of generators for $R(X)$ is given by $s_1, s_2, s_3, t_1, t_2, t_3,t$,
where $t\in H^0(h)$ with $t\not\in S$.
Among these generators there are relations of the following form:
\[
f_2(s_1s_2s_3,s_1t_2,s_2t_3,s_3t_1,t)=0,\ t_1t_2t_3+g_2(s_1s_2s_3,s_1t_2,s_2t_3,s_3t_1,t)=0,
\]
where the first relation comes from the equation of the quadric containing $X$, while 
the second relation comes from the equation of the cubic containing $X$ observing that
we can assume $\ell_1=s_1t_2, \ell_2=s_2t_3, \ell_3=s_3t_1$ and $x_4=s_1s_2s_3$.
It can be proved with the same type of argument used in the proof of 
Theorem \ref{S1} that the ideal $I$ is prime for general $f_2, g_2$. 
Thus $\cc[s_1,s_2,s_3,t_2,t_2,t_3,t]/I=R[X]$.
\end{proof}
\end{example}

 \begin{example}[Cases $S_{1,1,1}$ and $S_{1,1,2}$]\label{double}
K3 surfaces with Picard lattices 
 $U\oplus A_1\cong S_{1,1,1}$ and $U(2)\oplus A_1\cong S_{1,1,2}$ 
 carry a non-symplectic involution acting trivially on their Picard group. 
 A presentation of their Cox rings has been computed in \cite[Proposition 6.6]{A.H.L}
 (see Table \ref{TableGen}).
We observe that K3 surfaces with $\Cl(X)\cong U\oplus A_1$ 
are special for us since they are the only Mori dream K3 surfaces 
of Picard number three having an elliptic fibration with a section. 
In this case we can not apply  Test 1 and Test 2  in Section \ref{sec-gen},  
since the linear system associated to a nef divisor 
on such surface is not always base point free, see Proposition \ref{lsk3}.
\end{example}

 \section{Tables}\label{tables}
 This section contains the tables describing the effective and nef cone (Tables 1 and 2),
 and the degrees of a generating set of $R(X)$ (Table 3) for
 Mori dream K3 surfaces of Picard number three.
 We recall that $\Cl(X)$ denotes the Picard lattice of the surface,
 $E(X)$ is the set of generators of the extremal rays of the effective cone 
 (i.e. the set of classes of the $(-2)$-curves), 
 $\BEff(X)$ is the Hilbert  basis 
 of the effective cone, $N(X)$ is the set of generators of the 
 extremal rays of the nef cone and 
 $\BNef(X)$ is the Hilbert basis of the nef cone.
 
\newpage
 \ \\
 
\begin{longtable}{|c|c|c|c|c|c|c|}
 \hline
${N}^\circ$&$\Cl(X)$& $E(X)$& $\BEff(X)$ & $N(X)$ &  $\BNef(X)$ \\
\hline\hline
1& $S_1$ & 
 $\begin{array}{c}
(0,  1,  0),\\
    (0,  0,  1),\\
    (1,  -2,  0),\\
    (1,  0,  -2),\\
    (2,  -3,  -2),\\
    (2,  -2,  -3)
\end{array}$  & 
 $\begin{array}{c}
 E(X)\\
 \cup\\
 \{(1,-1,-1)\} 
 \end{array}$ & 
$\begin{array}{c}
( 1, 0,  0),\\
    ( 2,  -3, 0),\\
    ( 2,  0,  -3),\\
    ( 4,  -6,  -3),\\
    (4,  -3,  -6),\\
    (5,  -6,  -6)
\end{array}$& 
$\begin{array}{c}
( 1, -1,  -1), ( 1,  -1, 0),\\
    ( 1,  0,  -1), ( 1,  0,  0),\\
    (2,  -3,  -1), (2,  -3,  0),\\
(2, -1, -3), (2, 0, -3),\\
(3, -4, -3), (3, -3, -4),\\
(4, -6, -3), (4, -3, -6),\\
(5, -6, -6)
\end{array}$\\
\hline
2&$S_2$ & 
$\begin{array}{c}
(0,  1,  0),\\
    (0,  0,  1),\\
    (1, -5,  -3),\\
    (1,  -3, -5),\\
    (2, -9, -8),\\
    (2, -8, -9)
\end{array}$ &
$\begin{array}{c}
 E(X)\\
 \cup\\
 \{(1,-4,-4)\} 
 \end{array}$ & 
 $\begin{array}{c}
( 1,0,  0),\\
    ( 5,  -24, -12),\\
    ( 5,  -12,  -24),\\
    ( 13,  -60,  -48),\\
    (13,  -48, -60),\\
    (17,  -72,  -72)
\end{array}$& $N_{2}(X)$\\
\hline
3&$S_3$ & $\begin{array}{c}
(0,  1,  0),\\
    (0,  0,  1),\\
    (1, -3,  -2),\\
    (1,  -2, -3)
\end{array}$  &
$E(X)$ &
$\begin{array}{c}
( 1, 0,  0),\\
    ( 3,  -8, -4),\\
    ( 3,  -4,  -8),\\
    ( 5,  -12,  -12)
\end{array}$ & $\begin{array}{c}
(1, -2, -2), (1, -2, -1),\\
    (1, -1, -2), (1, -1, -1),\\
    (1, 0, 0), (2, -5, -4),\\
    (2, -5, -3), (2, -4, -5),\\
    (2, -3, -5), (3, -8, -4),\\
    (3, -7, -7), (3, -4, -8),\\
    (5, -12, -12)
\end{array}$ \\
\hline
4&$S_4$ & $\begin{array}{c}
(0,  1,  0),\\
    (0,  -1,  -1),\\
    (1, -1,  1),\\
    (2,  1,  3)
\end{array}$  &
$E(X)$ &
$\begin{array}{c}
( 3, -1,  -2),\\
    ( 7,  -9, 2),\\
    ( 13,  9,  18),\\
    ( 17,  1,  22)
\end{array}$& $\begin{array}{c}
(1, -1, 0), (1, 0, 0),\\
    (1, 0, 1), (2, -2, 1),\\
    (2, -1, -1), (2, 1, 2),\\
  (3, -1, -2), (3, -1, 3),\\
    (3, 1, 4), (3, 2, 4),\\
    (4, -5, 1), (4, 0, 5),\\
    (7, -9, 2), (8, 5, 11),\\
    (10, 1, 13), (13, 9, 18),\\
    (17, 1, 22)
\end{array}$\\
\hline
5&$S_5$ & $\begin{array}{c}
(0,  1,  0),\\
    (1,  -1,  -2),\\
    (0, 0,  1),\\
    (1,  -2,  -1)
\end{array}$ &
$E(X)$ &
$\begin{array}{c}
( 1, 0,  0),\\
    ( 3,  -4, -4),\\
    ( 3,  -4,  -2),\\
    ( 3,  -2,  -4)
\end{array}$ & $\begin{array}{c}
 (1, -1, -1), (1, 0, 0),\\
    (2, -2, -1), (2, -1, -2),\\
    (3, -4, -4), (3, -4, -3),\\
    (3, -4, -2), (3, -3, -4),\\
    (3, -2, -4)
\end{array}$ \\
\hline
6&$S_6$
  & $\begin{array}{c}
 (0,  -1,  0),\\
 (0,  0, -1),\\
    (1, 3,  1),\\
    (2, 3, 5),\\
(2,  5,  4),\\
    (3, 6, 7)
\end{array}$ &
$\begin{array}{c}
 E(X)\\
 \cup\\
 \{(1,2,2)\} 
 \end{array}$ & 
$\begin{array}{c}
( 3, -1,  -2),\\
    ( 5,  13, 4),\\
    ( 17,  31,  40),\\
    ( 19,  23,  46),\\
    (25,  65,  42),\\
    (41,  89,  90)
\end{array}$ & $N_{6}(X)$\\
\hline
\addlinespace
\addlinespace
\caption{Effective and Nef cone.}
\label{eff&nef}
\end{longtable}

\newpage
 \ \\
\addtocounter{table}{-1}
\begin{longtable}{|c|c|c|c|c|c|c|}
\hline
${N}^\circ$&$\Cl(X)$& $E(X)$& $\BEff(X)$ & $N(X)$ &  $\BNef(X)$ \\
\hline\hline
7& $S'_{4,1,2}$
 & $\begin{array}{c}
    (0, 1, 0),\\
    (2, 3, 1),\\
    (0, 1, 1),\\
    (2, 3, 2)
\end{array}$ & 
$\begin{array}{c}
 E(X)\\
 \cup\\
 \{(1,2,1)\} 
 \end{array}$ & 
$\begin{array}{c}
(0, 2, 1),\\
    (2, 4, 1),\\
    (2, 4, 3),\\
    (4, 6, 3)
\end{array}$& $\begin{array}{c}
(0, 2, 1), (1, 2, 1),\\
    (1, 3, 1), (1, 3, 2),\\
    (2, 4, 1), (2, 4, 3),\\
    (3, 5, 2), (3, 5, 3),\\
    (4, 6, 3)
\end{array}$\\
\hline
8&$S_{4,1,1}$
  & $\begin{array}{c}
(0,  1,  0),\\
    (0,  0,  1),\\
    (1, 3,  4)
\end{array}$ & 
$E(X)$ &
$\begin{array}{c}
( 0,1,  1),\\
    ( 1,  3, 5),\\
    ( 1,  4,  4)
\end{array}$& $\begin{array}{c}
(0, 1, 1),(1, 3, 5),\\
    (1, 4, 4),(1, 4, 5)
\end{array}$\\
\hline
9&$S_{5,1,1}$
  & $\begin{array}{c}
(0,  1,  0),\\
    (0,  0,  1),\\
    (1,  4,  5),\\
    (4, 15, 24)
\end{array}$  &
$E(X)$ &
$\begin{array}{c}
(0, 1, 1),\\
    (1, 5, 5),\\
    (4, 15, 25),\\
    (5, 19, 29)
\end{array}$ & $\begin{array}{c}
(0, 1, 1), (1, 4, 6),\\
    (1, 5, 5), (1, 5, 6),\\
    (2, 8, 13), (3, 12, 17),\\
    (4, 15, 25), (5, 19, 29),\\
    (5, 19, 30)
\end{array}$ \\
\hline
10&$S_{6,1,1}$
  & $\begin{array}{c}
(0,  0,  1),\\
    (1,  5,  6),\\
    (2,  9,  14),\\
    (0,1,0)
\end{array}$ &
$E(X)$ &
$\begin{array}{c}
(0, 1, 1),\\
    (1, 6, 6),\\
    (2, 9, 15),\\
    (3, 14, 20)
\end{array}$& $\begin{array}{c}
(0, 1, 1), (1, 5, 7),\\
    (1, 5, 8), (1, 6, 6),\\
    (1, 6, 7), (2, 9, 15),\\
    (2, 10, 13), (3, 14, 20),\\
    (3, 14, 21)
\end{array}$\\
\hline
11&$S_{7,1,1}$
 & $\begin{array}{c}
    (0,  0,  1),\\
    (1,  6,  7),\\
    (0,  1,  0),\\
    (3, 16, 24),\\
    (4, 21, 34),\\
    (6, 33, 46)
\end{array}$ &
$\begin{array}{c}
 E(X)\\
 \cup\\
 \{(2,11,16)\} 
 \end{array}$ & 
$\begin{array}{c}
 (0, 1, 1),\\
    (1, 7, 7),\\
    (4, 21, 35),\\
    (7, 37, 58),\\
    (7, 39, 53),\\
    (9, 49, 70)
\end{array}$& $\begin{array}{c}
(0, 1, 1), (1, 6, 8),\\
    (1, 6, 9), (1, 7, 7),\\
    (1, 7, 8), (2, 11, 16),\\
    (2, 11, 17), (2, 11, 18),\\
    (2, 12, 15), (3, 16, 25),\\
    (3, 16, 26), (3, 17, 23),\\
    (4, 21, 35), (4, 22, 31),\\
    (4, 23, 30), (5, 27, 40),\\
    (5, 28, 38), (6, 32, 49),\\
    (7, 37, 58), (7, 37, 59),\\
    (7, 38, 55), (7, 39, 53),\\
    (8, 43, 64), (9, 49, 70),\\
    (10, 55, 77)
\end{array}$\\
\hline
12&$S_{8,1,1}$  & $\begin{array}{c}
    (0, 0, 1),\\
    (1, 6, 9),\\
    (1, 7, 8),\\
    (0, 1, 0)
\end{array}$  & 
$E(X)$ &
$\begin{array}{c}
(0, 1, 1),\\
    (1, 6, 10),\\
    (1, 8, 8),\\
    (2, 13, 17)
\end{array}$& $\begin{array}{c}
 (0, 1, 1), (1, 6, 10),\\
    (1, 7, 9), (1, 7, 10),\\
    (1, 8, 8), (1, 8, 9),\\
    (2, 13, 17), (2, 13, 18),\\
    (2, 14, 17)
\end{array}$\\
\hline
\addlinespace
\addlinespace
\caption{Effective and Nef cone.}
\end{longtable}

 \newpage
 \ \\
  
\addtocounter{table}{-1}
\begin{longtable}{|c|c|c|c|c|c|c|}
 \hline
${N}^\circ$&$\Cl(X)$& $E(X)$& $\BEff(X)$ & $N(X)$ &  $\BNef(X)$ \\
\hline\hline
13&$S_{10,1,1}$
 & $\begin{array}{c}
    (1,  9, 10),\\
    (4, 32, 43),\\
    (0,  0,  1),\\
    (3, 23, 34),\\
    (2, 15, 24),\\
    (6, 47, 66),\\
    (4, 33, 42),\\
    (0,  1,  0)
\end{array}$ &  
$\begin{array}{c}
 E(X)\\
 \cup\\
 \{(1,8,11)\} 
 \end{array}$ & 
$\begin{array}{c}
(0, 1, 1),\\
    (1, 10, 10),\\
    (2, 15, 25),\\
    (5, 38, 58),\\
    (5, 42, 52),\\
    (8, 65, 85),\\
    (9, 70, 100),\\
    (10, 79, 109)
\end{array}$& $\begin{array}{c}
(0, 1, 1), (1, 8, 11),\\
    (1, 8, 12), (1, 8, 13),\\
    (1, 9, 11), (1, 10, 10),\\
    (1, 10, 11), (2, 15, 25),\\
    (2, 17, 21), (3, 23, 35),\\
    (3, 26, 31), (4, 31, 45),\\
    (5, 38, 58), (5, 38, 59),\\
    (5, 40, 54), (5, 41, 53),\\
    (5, 42, 52), (7, 54, 79),\\
    (7, 55, 77), (8, 65, 85),\\
    (9, 70, 100), (9, 72, 97),\\
    (9, 74, 95), (10, 79, 109),\\
    (13, 102, 143)
\end{array}$\\
\hline
14&$S_{12,1,1}$
& $\begin{array}{c}
(0,  1,  0),\\
    (1,  9, 14),\\
    (2, 19, 26),\\
    (0,  0,  1),\\
    (1, 11, 12),\\
    (2, 20, 25)
\end{array}$ &
$\begin{array}{c}
 E(X)\\
 \cup\\
 \{(1,10,13)\} 
 \end{array}$ & 
$\begin{array}{c}
(0, 1, 1),\\
    (1, 9, 15),\\
    (1, 12, 12),\\
    (3, 28, 40),\\
    (3, 31, 37),\\
    (4, 39, 51)
\end{array}$ & $\begin{array}{c}
(0, 1, 1), (1, 9, 15),\\
    (1, 10, 13), (1, 10, 14),\\
    (1, 10, 15), (1, 11, 13),\\
    (1, 12, 12), (1, 12, 13),\\
    (2, 19, 27), (2, 21, 25),\\
    (3, 28, 40), (3, 28, 41),\\
    (3, 29, 39), (3, 30, 38),\\
    (3, 31, 37), (3, 32, 37),\\
    (4, 39, 51), (5, 48, 65),\\
    (5, 50, 63)
\end{array}$ \\
\hline
15&$S_{1,2,1}$
  & $\begin{array}{c}
(1, 0, 0),\\
    (0, 0, 1),\\
    (0, 1, 1)
\end{array}$  &
$E(X)$ &
$\begin{array}{c}
    (0, 1, 2),\\
    (4, 3, 8),\\
    (4, 5, 8)
\end{array}$& $\begin{array}{c}
(0, 1, 2), (1, 1, 2),\\
    (2, 2, 5), (2, 3, 5),\\
    (4, 3, 8), (4, 5, 8)
\end{array}$\\
\hline
16&$S_{1,3,1}
$
 & $\begin{array}{c}
          (1, 0, 0),\\
    (0, 0, 1),\\
    (0, 1, 2)
\end{array}$ &
$E(X)$ &
$\begin{array}{c}
    (0, 1, 3),\\
    (1, 1, 2),\\
    (2, 1, 4)
\end{array}$& $\begin{array}{c}
(0, 1, 3), (1, 1, 2),\\
    (1, 1, 3), (2, 1, 4)
\end{array}$\\
\hline
17&$S_{1,4,1}$  & $\begin{array}{c}
     (1, 0, 0),\\
    (0, 1, 3),\\
    (0, 0, 1),\\
    (3, 3, 8)
\end{array}$ & 
$E(X)$ &
$\begin{array}{c}
(0, 1, 4),\\
    (4, 3, 8),\\
    (8, 3, 16),\\
    (24, 29, 80)
\end{array}$& $\begin{array}{c}
(0, 1, 4), (1, 1, 3),\\
    (1, 1, 4), (2, 1, 4),\\
    (2, 1, 5), (2, 3, 9), \\
    (3, 2, 6), (4, 3, 8),\\
    (4, 4, 11), (4, 5, 14),\\
    (5, 2, 10), (7, 8, 22),\\
    (8, 3, 16), (14, 17, 47),\\
    (24, 29, 80)
\end{array}$\\
\hline
\addlinespace
\addlinespace
\caption{Effective and Nef cone.}
\end{longtable}

\newpage
 \ \\
  
\addtocounter{table}{-1}
\begin{longtable}{|c|c|c|c|c|c|c|}
\hline
${N}^\circ$&$\Cl(X)$& $E(X)$& $\BEff(X)$ & $N(X)$ &  $\BNef(X)$ \\
\hline\hline
18&$S_{1,5,1}$
 & $\begin{array}{c}
( 1,  0,  0),\\
    ( 0,  1,  4),\\
    ( 0,  0,  1),\\
    ( 4,  3, 10),\\
    ( 5,  6, 21),\\
    (16, 16, 55)
\end{array}$ & 
$\begin{array}{c}
 E(X)\\
 \cup\\
 \{(2,2,7)\} 
 \end{array}$ &
$\begin{array}{c}
(0, 1, 5),\\
    (5, 3, 10),\\
    (5, 7, 25),\\
    (10, 3, 20),\\
    (20, 19, 65),\\
    (190, 197, 680)
\end{array}$ & $\begin{array}{c}
 (0, 1, 5), (1, 1, 4),\\
    (1, 1, 5), (1, 2, 8),\\
    (2, 1, 4), (2, 1, 5),\\
    (2, 1, 6), (2, 2, 7),\\
    (2, 3, 11), (3, 1, 6),\\
    (3, 2, 7), (5, 3, 10),\\
    (5, 7, 25), (6, 2, 13),\\
    (6, 5, 17), (7, 8, 28),\\
    (9, 9, 31), (10, 3, 20),\\
    (13, 12, 41), (14, 15, 52),\\
    (20, 19, 65), (25, 25, 86),\\
    (30, 31, 107), (46, 47, 162),\\
    (51, 53, 183), (118, 122, 421),\\
    (190, 197, 680)
\end{array}$ \\
\hline
19&$S_{1,6,1}$  & $\begin{array}{c}
(1, 0, 0),\\
    (1, 1, 4),\\
    (0, 0, 1),\\
    (0, 1, 5)
\end{array}$ & 
$E(X)$ &
$\begin{array}{c}
    (0, 1, 6),\\
    (2, 1, 4),\\
    (4, 1, 8),\\
    (4, 7, 32)
\end{array}$& $\begin{array}{c}
(0, 1, 6), (1, 1, 5),\\
    (1, 1, 6), (1, 2, 10),\\
    (2, 1, 4), (2, 1, 5),\\
    (2, 1, 6), (2, 1, 7),\\
    (2, 2, 9), (2, 3, 14),\\
    (2, 4, 19), (3, 1, 6),\\
    (3, 1, 7), (3, 4, 18),\\
    (3, 5, 23), (4, 1, 8),\\
    (4, 7, 32)
\end{array}$\\
\hline
20&$S_{1,9,1}$
 & $\begin{array}{c}
 (2,  1,  6),\\
(0,  1,  8),\\
    (0,  0,  1),\\
    (1,  0,  0),\\
    (4,  5, 34),\\
    (3,  6, 43),\\
    (7,  6, 39),\\
    (5,  8, 56),\\
    (8,  8, 53)
\end{array}$ & 
$\begin{array}{c}
 E(X)\\
 \cup\\
 \{(1,1,7),(2,3,21), (3,3,20)\} 
 \end{array}$ &
$\begin{array}{c}
(0, 1, 9),\\
    (3, 1, 6),\\
    (3, 7, 51),\\
    (6, 1, 12),\\
    (9, 7, 45),\\
    (9, 13, 90),\\
    (12, 13, 87),\\
    (42, 73, 516),\\
    (78, 73, 480)
\end{array}$ & $N_{20}(X)$  \\
\hline
21&$S_{1,1,1}$
  & $\begin{array}{c}
(-1,  0,  0),\\
    (0,  1,  0),\\
    (1,  0,  1)
\end{array}$ & 
$E(X)$ &
$\begin{array}{c}
( 1, 1,  1),\\
    ( 1,  2, 2),\\
    ( 2,  3,  4)
\end{array}$& $\begin{array}{c}
( 1, 1,  1),\\
    ( 1,  2, 2),\\
    ( 2,  3,  4)
\end{array}$\\
\hline
22&$S_{1,1,2}$
  & $\begin{array}{c}
(0,  -1, - 1),\\
    (1,  2,  2),\\
    (0, 3,  2)
\end{array}$ & 
$E(X)$ &
$\begin{array}{c}
( 0,2,  1),\\
    ( 1,  1, 1),\\
    ( 1,  4,  3)
\end{array}$& $\begin{array}{c}
( 0,2,  1),\\
    ( 1,  1, 1),\\
    ( 1,  4,  3)
\end{array}$\\
\hline
\addlinespace
\addlinespace
\caption{Effective and Nef cone.}
\end{longtable}

\newpage
 
  \ \\

\addtocounter{table}{-1}
\begin{longtable}{|c|c|c|c|c|c|c|}
\hline
${N}^{\circ }$&$\Cl(X)$& $E(X)$& $\BEff(X)$& $N(X)$ &  $\BNef(X)$ \\
\hline\hline
23&$S_{1,1,3}$
  & $\begin{array}{c}
(0,  -2,  -1),\\
    (1,  6, 3),\\
    (2, 3,  2),\\
    (0,  5, 2)
\end{array}$ & 
$E(X)$ &
$\begin{array}{c}
( 0, 3,  1),\\
    ( 3,  9, 5),\\
    ( 3,  24,  11),\\
    ( 6,  12,  7)
\end{array}$ & $\begin{array}{c}
(0, 3, 1), (1, 4, 2),\\
    (1, 9, 4), (2, 6, 3),\\
    (3, 7, 4), (3, 9, 5),\\
    (3, 14, 7), (3, 19, 9),\\
    (3, 24, 11), (4, 9, 5),\\
    (6, 12, 7)
\end{array}$ \\
\hline

24&$S_{1,1,4}$
 & $\begin{array}{c}
(0,  -3, - 1),\\
    (1,  2,  1),\\
    (0, 7, 2),\\
    (1,  12,  4)
\end{array}$ & 
$E(X)$ &
$\begin{array}{c}
( 0, 4,  1),\\
    ( 2,  8, 3),\\
    ( 2,  14,  5),\\
    ( 2,  28,  9)
\end{array}$& $\begin{array}{c}
(0, 4, 1), (1, 6, 2),\\
    (1, 9, 3), (1, 16, 5),\\
    (2, 8, 3), (2, 11, 4),\\
    (2, 14, 5), (2, 21, 7),\\
    (2, 28, 9)
\end{array}$\\
\hline
25&$S_{1,1,6}$
 & $\begin{array}{c}
(0,  5,  1),\\
    (2,  3,  1),\\
    (3, 16, 4),\\
    (1,  0,  0),\\
(4,15,4),\\
(0,1,0)
\end{array}$ & 
$\begin{array}{c}
 E(X)\\
 \cup\\
 \{(1,4,1)\} 
 \end{array}$ &
$\begin{array}{c}
( 0, 6,  1),\\
    ( 3,  3, 1),\\
    ( 3,  6,  1),\\
    ( 3,  21,  5),\\
(6,18,5),\\
(15,66,17)
\end{array}$& 
$N_{25}(X)$
\\
\hline
26&$S_{1,1,8}$
& $\begin{array}{c}
 (3,  4,  1),\\
    (1,  0,  0),\\
    (4,  9,  2),\\
    (0,  1,  0),\\
    (4, 15,  3),\\
    (0,  7,  1),\\
    (3, 16,  3),\\
    (1, 12,  2)
\end{array}$  &
$\begin{array}{c}
 E(X)\\
 \cup\\
 \{(1,6,1),(2,5,1),(2,11,2),(3,10,2)\} 
 \end{array}$ &
$\begin{array}{c}
(0, 8, 1),\\
    (4, 4, 1),\\
    (4, 8, 1),\\
    (4, 28, 5),\\
    (4, 56, 9),\\
    (8, 24, 5),\\
    (20, 40, 9),\\
    (20, 88, 17)
\end{array}$ & $N_{26}(X)$  \\
\hline
\addlinespace
\addlinespace
\caption{Effective and Nef cone.}
\end{longtable}

\newpage
\vspace{0.6cm}

\renewcommand{\arraystretch}{1.1}\setlength{\tabcolsep}{2pt}
\begin{longtable}{|c|c|}
\hline 
 ${N}^\circ$& $\BNef(X)$ \\
\hline\hline
2&$\begin{array}{c}(1, -4, -4), (1, -4, -3), (1, -4, -2), (1, -3, -4), (1, -3, -3), (1, -3, -2),\\
 (1, -2, -4), (1, -2, -3), (1, -2, -2), (1, -2, -1), (1, -1, -2), (1, -1, -1), (1, 0, 0),\\ 
 (2, -9, -7),  (2, -9, -6), (2, -9, -5), (2, -7, -9), (2, -6, -9),(2, -5, -9),(3, -14, -10),\\
  (3, -14, -9),  (3, -14, -8), (3, -14, -7),  (3, -13, -12),(3, -12, -13), (3, -10, -14),\\
   (3, -9, -14),  (3, -8, -14), (3, -7, -14), (4, -19, -11),(4, -19, -10), (4, -18, -15), \\
   (4, -15, -18), (4, -11, -19), (4, -10, -19), (5, -24, -12),(5, -23, -18), (5, -22, -20),\\
    (5, -21, -21), (5, -20, -22), (5, -18, -23), (5, -12, -24),(6, -27, -23), (6, -23, -27), \\
    (7, -32, -26), (7, -30, -29), (7, -29, -30),  (7, -26, -32),(8, -37, -29), (8, -29, -37), \\
    (9, -41, -34), (9, -38, -38), (9, -34, -41),  (10, -46, -37),(10, -37, -46), (11, -47, -46),\\
     (11, -46, -47), (13, -60, -48), (13, -55, -55), (13, -48, -60),(17, -72, -72)\end{array}$\\
     \hline
    6
 & $\begin{array}{c}
    (1, 0, 0), (1, 1, 0), (1, 1, 1), (1, 1, 2), (1, 2, 1), (1, 2, 2),(2, 0, -1), (2, 4, 1),\\
    (2, 5, 2), (2, 5, 3),(3, -1, -2), (3, 4, 7), (3, 5, 7), (3, 7, 2),(3, 7, 6), (4, 8, 9),\\
    (4, 10, 3), (4, 10, 7), (5, 6, 12), (5, 7, 12), (5, 13, 4), (5, 13, 5),(5, 13, 6), (5, 13, 7),\\
    (5, 13, 8), (6, 11, 14), (6, 13, 13), (7, 18, 12), (8, 13, 19), (9, 19, 20), (11, 19, 26), \\
    (11, 22, 25),(11, 24, 24), (12, 15, 29),(13, 25, 30), (14, 25, 33), (15, 28, 35), \\
    (15, 39, 25),(17, 31, 40), (19, 23, 46),(25, 54, 55), (25, 65, 42),(41, 89, 90)
\end{array}$\\
\hline
 20& $\begin{array}{c}(0, 1, 9),(1, 1, 7), (1, 1, 8), (1, 1, 9),(1, 2, 15), (1, 3, 23), (2, 1, 7), (2, 1, 8),\\
    (2, 1, 9),
    (2, 1, 10),
    (2, 3, 21),
    (2, 4, 29), (2, 5, 37),(3, 1, 6), (3, 1, 7), (3, 1, 8),(3, 1, 9),\\
    (3, 1, 10),(3, 2, 13),
    (3, 3, 20),
    (3, 7, 51),
    (4, 1, 8),
    (4, 1, 9),
    (4, 1, 10),
    (4, 1, 11),
    (5, 1, 10),\\
    (5, 1, 11),
    (5, 3, 19),
    (5, 4, 26),
    (5, 9, 64),
    (6, 1, 12),
    (6, 8, 55),
    (7, 5, 32),
    (7, 8, 54),
    (7, 11, 77),\\
    (8, 15, 107),
    (9, 7, 45),
    (9, 13, 90),
    (10, 9, 59),
    (10, 13, 89),
    (10, 17, 120),
    (11, 11, 73),
    (11, 13, 88),\\
    (12, 13, 87),
    (12, 19, 133),
    (13, 23, 163),
    (15, 25, 176),
    (16, 29, 206),
    (17, 15, 98),
    (18, 17, 112),\\
    (18, 31, 219),
    (19, 19, 126),
    (20, 33, 232),
    (21, 37, 262),
    (23, 39, 275),
    (25, 23, 151),
    (26, 25, 165),\\
    (26, 45, 318),
    (29, 51, 361),
    (31, 53, 374),
    (32, 29, 190),
    (33, 31, 204),
    (34, 33, 218),
    (34, 59, 417),\\
    (40, 37, 243),
    (41, 39, 257),
    (42, 73, 516),
    (48, 45, 296),
    (55, 51, 335),
    (56, 53, 349),
    (63, 59, 388),\\
    (78, 73, 480)\end{array}$\\
    \hline 
  25 &
   $\begin{array}{c} (0, 6, 1),
    (1, 4, 1),
    (1, 5, 1),
    (1, 6, 1),
    (1, 9, 2),
    (2, 4, 1),
    (2, 5, 1),
    (2, 6, 1),\\
    (3, 3, 1),
    (3, 4, 1),
    (3, 5, 1),
    (3, 6, 1),
    (3, 7, 2),
    (3, 21, 5),
    (4, 20, 5),
    (5, 19, 5),
    (6, 18, 5),\\
    (7, 36, 9),
    (8, 35, 9),
    (9, 34, 9),
    (11, 51, 13),
    (12, 50, 13),
    (15, 66, 17)
\end{array}$\\
\hline
26&
 $\begin{array}{c} (0, 8, 1),
    (1, 6, 1),
    (1, 7, 1),
    (1, 8, 1),
    (1, 13, 2),
    (1, 20, 3),
    (2, 5, 1),\\
    (2, 6, 1),
    (2, 7, 1),
    (2, 8, 1),
    (2, 11, 2),
    (2, 18, 3),
    (2, 25, 4),
    (2, 32, 5),
    (3, 5, 1),\\
    (3, 6, 1),
    (3, 7, 1),
    (3, 8, 1),
    (3, 10, 2),
    (3, 23, 4),
    (3, 30, 5),
    (3, 37, 6),
    (3, 44, 7),\\
    (4, 4, 1),
    (4, 5, 1),
    (4, 6, 1),
    (4, 7, 1),
    (4, 8, 1),
    (4, 28, 5),
    (4, 35, 6),
    (4, 42, 7),\\
    (4, 49, 8),
    (4, 56, 9),
    (5, 9, 2),
    (5, 27, 5),
    (6, 14, 3),
    (6, 26, 5),
    (7, 19, 4),
    (7, 25, 5),\\
    (8, 13, 3),
    (8, 24, 5),
    (8, 43, 8),
    (9, 18, 4),
    (9, 42, 8),
    (10, 23, 5),
    (10, 41, 8),
    (11, 28, 6),\\
    (11, 40, 8),
    (12, 22, 5),
    (12, 58, 11),
    (13, 27, 6),
    (13, 57, 11),
    (14, 32, 7),
    (14, 56, 11),\\
    (16, 31, 7),
    (16, 73, 14),
    (17, 36, 8),
    (17, 72, 14),
    (20, 40, 9),
    (20, 88, 17)
\end{array}$\\
\hline  
\addlinespace
\addlinespace
\caption{$\BNef(X)$ for $\Cl(X)=S_2$, $S_{6}$, $S_{1,9,1}$, $S_{1,1,6}$ and $S_{1,1,8}$.}
\end{longtable}

 \newpage
 \ \\

\renewcommand{\arraystretch}{1.4}\setlength{\tabcolsep}{2pt}
\begin{longtable}{|c|c|l|}
  \hline
${N}^\circ$&$\Cl(X)$&  Degrees of generators of $R(X)$\\
\hline
\hline
1&$S_1$&$\BEff$\\
\hline
2&$S_2$&$E, \BNef$
\\
\hline
3&$S_3$&$E, \BNef$\\
\hline
4&$S_4$&$E, \BNef$\\ 
\hline
5&$S_5$&$E, \BNef[i], i=1,2,5$\\
\hline
6&$S_6$& 
 $E, \BNef[i], i=1-7, 10, 11^*,12,13, 15, 16, 18-20, 25, 27, 28, 30, 33, 
34, 38, 40^*, 41, 42^*,43^*$\\
\hline
7&$S'_{4,1,2}$&$E, \BNef $\\
\hline
8&$S_{4,1,1}$&$E. \BNef$\\
\hline
9&$S_{5,1,1}$&$E, \BNef$\\
\hline
10&$S_{6,1,1}$&$E, \BNef$\\
\hline
11&$S_{7,1,1}$&$E, \BNef$\\
\hline
12&$S_{8,1,1}$&$E, \BNef$\\
\hline
13&$S_{10,1,1}$&{$\!\begin{aligned}
&E, \BNef, 3\BNef[2]\\
& \BNef[2]+\BNef[i],\ i= 1^*,6^*,8^*,13^*,17^*,20^*,21^*,24^*\\
\end{aligned}$}\\
  \hline
14&$S_{12,1,1}$&{$\!\begin{aligned}
&E\cup \BNef\\
\end{aligned}$}\\
\hline
15&$S_{1,2,1}$&$E, \BNef$\\
\hline
16&$S_{1,3,1}$&$E, \BNef$\\
\hline
17&$S_{1,4,1}$&$E, \BNef[i], i =1-12,13^*,14,15^*$\\
\hline
18&$S_{1,5,1}$&$E, \BNef[i], i=1-17,18^*,19-26,27^*$\\
\hline
19&$S_{1,6,1}$&$E,\BNef$\\
\hline
20&$S_{1,9,1}$&{$\!\begin{aligned}
& E, \BNef, \BNef[11]+\BNef[i],  i=20^*,34^*,43^*\\
&\BNef[2]+\BNef[i], i =11^*,12^*,13^*,20^*,29^*,33^* \\
\end{aligned}$}\\
\hline
21&$S_{1,1,1}$&$E, \BNef[1],\BNef[3], \BNef[2]+\BNef[3]$\\
\hline
22&$S_{1,1,2}$&
$E, \BNef[1],\BNef[2], \BNef[1]+\BNef[2]+\BNef[3]$\\
\hline
23&$S_{1,1,3}$&$E, \BNef[1],\BNef[2],\BNef[6]$\\
\hline
24&$S_{1,1,4}$&$E, \BNef[1],\BNef[3], \BNef[7]$\\
\hline
25&$S_{1,1,6}$&{$\!\begin{aligned}
&E, \BNef[i],\ i= 1,2,3,5,6,9,13-17\\
\end{aligned}$}\\
\hline
26&$S_{1,1,8}$&{$\!\begin{aligned}
&E, \BNef[i],\ i=1,2,3,5,7,11,15,19,20,24,29,34,35,38,39,41\\
&\BNef[2]+\BNef[i],\ i =7^*,11^*,19^*\quad \BNef[19]+\BNef[i],\ i=7^*,11^*
\end{aligned}$}\\
\hline
\addlinespace
\caption{Degrees of a set of generators of $R(X)$}\label{TableGen}
\end{longtable}

\bibliography{ref}
\bibliographystyle{alpha}

\end{document}